\documentclass[a4paper]{article}
\usepackage[utf8]{inputenc}
\usepackage{amsthm,amsmath,amssymb,amsfonts}
\usepackage{tikz}
\usepackage{xcolor}
\usepackage{palatino}
\usepackage{hyperref}

\newtheorem{theorem}{Theorem}
\newtheorem{lemma}[theorem]{Lemma}
\newtheorem{proposition}[theorem]{Proposition}

\theoremstyle{definition}
\newtheorem{definition}{Definition}

\newcommand{\tuple}[1]{{\langle #1 \rangle}}
\newcommand{\ax}[1]{{\rm #1}}

\newcommand{\R}{{\sqsubseteq}}

\newcommand{\SF}{\mathsf{S4}}
\newcommand{\csf}{\mathsf{CS4}}
\newcommand{\isf}{\mathsf{IS4}}
\newcommand{\gsf}{\mathsf{GS4}}
\newcommand{\gsfc}{\mathsf{GS4^c}}

\newcommand{\mcsf}{{M_\infty^\csf}}
\newcommand{\mgsf}{{M_\infty^\gsf}}
\usepackage[citestyle=numeric, bibstyle=numeric, maxbibnames=5, date=year]{biblatex}
\bibliography{biblio}

\title{Modalities in non-classical variations of S4}
\author{Leonardo Pacheco}

\begin{document}
\maketitle

\begin{abstract}
    A classical result in modal logic states that $\mathsf{S4}$ has $14$ modalities, that is, every sequence of negations, boxes, and diamonds is equivalent to one in a set of $14$ such sequences.
    We study analogous results for the non-classical analogues $\mathsf{CS4}$, $\mathsf{IS4}$, $\mathsf{GS4}$, and $\mathsf{GS4^c}$ of $\mathsf{S4}$.
    First, we show that, while all these logics have finitely many $\{\Box,\Diamond\}$- and $\{\neg,\Box\}$-modalities, the logic $\mathsf{CS4}$ has infinitely many $\{\neg,\Diamond\}$-modalities.
    Second, we show that $\mathsf{IS4}$ and $\mathsf{GS4}$ have finitely many $\{\neg,\Diamond\}$-modalities, but they have infinitely many $\{\neg,\Box,\Diamond\}$-modalities.
    At last, we show that $\mathsf{GS4^c}$ has finitely many $\{\neg,\Box,\Diamond\}$-modalities.
\end{abstract}

%%%%%%%%%%%%%%%%%%%%%%%%%%%%%%%%%%%%%%%%%%%%%%%%%%%%%%%%%%%%%%%%%%%%%%%%%%%%%%%%
\section{Introduction}
%%%%%%%%%%%%%%%%%%%%%%%%%%%%%%%%%%%%%%%%%%%%%%%%%%%%%%%%%%%%%%%%%%%%%%%%%%%%%%%%

It is widely known that, over $\SF$, every sequence of negations, boxes, and diamonds is equivalent to one of:
\[
    \varepsilon, \Box, \Diamond, \Box\Diamond, \Diamond\Box, \Box\Diamond\Box, \Diamond\Box\Diamond, 
    \neg, \neg\Box, \neg\Diamond, \neg\Box\Diamond, \neg\Diamond\Box, \neg\Box\Diamond\Box \text{, and } \neg\Diamond\Box\Diamond.
\] 
This result was first proved by Parry~\cite{parry1939modalities} in 1939; it has been widely studied \cite{bellissima1989modalities,vanbenthem1976reduction,fitch1973reduction,chellas1980textbook} and is now a standard exercise in modal logic textbooks.
A topological analogue was also studied by Kuratowski \cite{kuratowski1922closure}; see also~\cite{mckinsey1944topology,gardner2008kuratowski}.

We prove (and refute) analogous results over the following non-classical variations of $\SF$:
\[
    \csf \subseteq \isf \subseteq \gsf \subseteq \gsfc.
\]
In particular, we prove that:
\begin{itemize}
    \item $\csf$ has finitely many $\{\Box,\Diamond\}$- and $\{\neg,\Box\}$-modalities but infinitely many $\{\neg,\Diamond\}$-modalities;
    \item $\isf$ and $\gsf$ have finitely many $\{\neg,\Diamond\}$-modalities but infinitely many $\{\neg,\Box, \Diamond\}$-modalities; and
    \item $\gsfc$ has finitely many $\{\neg,\Box,\Diamond\}$-modalities.
\end{itemize}
While interesting by themselves, these results also offer one explanation of why proving the finite model property for these logics is so hard: the proofs for $\csf$, $\gsf$, and $\gsfc$ use a relatively complicated filtration method~\cite{balbiani2021variants,balbiani2024variants}, and the published proof for $\isf$~\cite{girlando2023is4} contains a mistake~\cite{girlando2026is4mistake}.

\paragraph{Background.}
The logics we study belong to wider families of modal logics based on intuitionistic propositional logic and using both boxes and diamonds.
The logic $\csf$ was first studied by Alechina \emph{et al.}~\cite{alechina2001cs4} and is based on Wijesekera's constructive modal logic $\mathsf{CK}$~\cite{mendler2005constructive, wijesekera1990constructive}.
The intuitionistic modal logic $\isf$ was first studied by Fischer Servi \cite{servi1984axiomatizations} and later by Simpson\cite{simpson1994phdthesis}.
The logics $\gsf$ and $\gsfc$ are based on the Gödel--Dummett family of fuzzy logics; they were first studied by Caicedo and Rodríguez~\cite{caicedo2015bimodalgodellogic} and Rodríguez and Vidal~\cite{rodriguez2021crispgodellogic}.

As far as we are aware, there are no results in the literature about modalities for intuitionistic modal logics using both boxes and diamonds.
For logics using only boxes or only  diamonds, a few results are known.
Do\v sen~\cite{dosen1985stronger} and Font~\cite{font1986modality} proved that the $\Diamond$-free logic $\mathsf{iS4}$\footnote{$\mathsf{iS4}$ is called $\mathsf{HK4}_\Box$ by Do\v sen and $\mathsf{IM4}$ by Font.} has finitely many $\{\neg,\Diamond\}$-modalities.
Do\v sen~\cite{dosen1985stronger} also proved that the $\Box$-free logic $\mathsf{HS4}_\Diamond$ has infinitely many $\{\neg,\Diamond\}$-modalities.

Another related line of work can be found Ciraulo's paper~\cite{ciraulo2025constructivekuratowski}, which studies Kuratowski's topological problem in a constructive setting.
Do note that our situation is quite different from the one in that paper. 
From a metamathematical point of view, we work in a classical framework. 
From a topological point of view, the semantics for $\csf$, $\isf$, $\gsf$, and $\gsfc$ involve two topologies instead of one~\cite{aguilera2026polytopological}.

%%%%%%%%%%%%%%%%%%%%%%%%%%%%%%%%%%%%%%%%%%%%%%%%%%%%%%%%%%%%%%%%%%%%%%%%%%%%%%%%
\section{Preliminaries}
%%%%%%%%%%%%%%%%%%%%%%%%%%%%%%%%%%%%%%%%%%%%%%%%%%%%%%%%%%%%%%%%%%%%%%%%%%%%%%%%

\paragraph{Syntax.}
We will work within a standard language of intuitionistic modal logic, where, unlike classical modal logic, $\Box$ and $\Diamond$ are not interdefinable.
Fix a set $\mathrm{Prop}$ of propositional symbols.
The \emph{modal formulas} are defined by the following grammar:
\[
    \varphi := \bot \mid p\mid \varphi\land\varphi  \mid \varphi\lor\varphi \mid \varphi\to\varphi \mid \Box\varphi \mid \Diamond\varphi.
\]
As usual, we define $\neg\varphi:= \varphi\to\bot$ and $\varphi\leftrightarrow\psi := (\varphi \to \psi) \land (\psi\to\varphi)$.

\begin{definition}
    A \emph{(modal) logic} is a set of formulas closed under necessitation and \emph{modus ponens}:
        \[
            \ax{Nec}:=\frac{\varphi}{\Box\varphi} \text{ and } \ax{MP} :=\frac{\varphi \;\;\; \varphi\to\psi}{\psi}.
        \]
    We write $\Lambda\vdash\varphi$ for $\varphi\in\Lambda$.
    $\csf$ is the least logic containing all intuitionistic tautologies, and the axioms:
    \begin{itemize}
        \item ${\ax K} := (\Box(\varphi\to\psi) \to (\Box\varphi \to \Box \psi)) \land (\Box(\varphi\to\psi) \to (\Diamond\varphi \to \Diamond \psi))$;
        \item ${\ax 4} := (\Box\varphi \to \Box\Box\varphi) \land (\Diamond\Diamond\varphi \to \Diamond \varphi)$;
        \item ${\ax T} := (\Box\varphi \to \varphi) \land (\varphi \to \Diamond \varphi)$.
    \end{itemize}
\end{definition}

All other logics we consider are extensions of $\csf$, with combinations of the following axioms:
\begin{itemize}
    \item $\ax{N_\Diamond} := \neg\Diamond\bot$;
    \item $\ax{FS} := (\Diamond \varphi \to \Box\psi) \to \Box(\varphi\to\psi)$;
    \item $\ax{DP} := \Diamond (\varphi\lor\psi) \to \Diamond\varphi\lor\Diamond\psi$;
    \item $\ax{RV} := \Box(\varphi\lor\psi) \to \Box\varphi\lor\Diamond\psi$;
    \item $\ax{GD} := (\varphi\to\psi) \lor (\psi\to\varphi)$.
\end{itemize}
\noindent Above, \ax{FS} stands for {\em Fischer Servi,} \ax{DP} for {\em disjunctive possibility,} \ax{RV} for {\em Rodríguez--Vidal}, and \ax{GD} for {\em G\"odel--Dummett}.

\begin{definition}
    If $\Lambda$ is a logic and $\ax{X}$ is an axiom schema, then $\Lambda + \ax{X}$ is the least logic containing both $\Lambda$ and $\ax{X}$.
    We define the following extensions of $\csf$:
    \begin{itemize}
        \item $\isf  := \csf + \{\ax{N_\Diamond}, \ax{FS}, \ax{DP} \}$;
        \item $\gsf  := \isf + \{\ax{GD}\}$; and
        \item $\gsfc := \gsf + \{\ax{RV}\}$.
    \end{itemize}
    Above, $\sf c$ stands for `crisp' according to its use in fuzzy logic~\cite{rodriguez2021crispgodellogic}.
\end{definition}

\paragraph{Birelational semantics.}
Let $R$ be a relation over a set $W$.
We say $R$ is a \emph{preorder} iff $R$ is transitive and reflexive.
If $A\subseteq W$, then $A$ is an \emph{$R$-upset} iff $w\in A$ and $wRv$ implies $v\in A$; similarly $A$ is an \emph{$R$-downset} iff $w\in A$ and $vRw$ implies $v\in A$.
We say $A$ is an \emph{$R$-updownset} iff $A$ is both an $R$-upset and an $R$-downset.

\begin{definition}
    \label{def::csf-model}
    A birelational $\csf$-model is a tuple $M=\tuple{W, W^\bot, \preceq, \R, V}$, where:
    \begin{itemize}
        \item $W$ is the set of \emph{possible worlds};
        \item $W^\bot$ is the set of \emph{fallible worlds};
        \item the \emph{intuitionistic relation} $\preceq$ is a preorder over $W$;
        \item the \emph{modal relation} $\R$ is a preorder over $W$; and
        \item $V:\mathrm{Prop}\to \mathcal{P}(W)$ is a \emph{valuation function}.
    \end{itemize}
    We require that:
    \begin{itemize}
        \item $M$ is \emph{backward confluent}, that is, $w \R v \preceq v'$ implies there is $w'$ such that $w \preceq w'$ and $w' \R v'$;
        \item for all $P\in\mathrm{Prop}$, $V(P)$ is a $\preceq$-upset and $W^\bot\subseteq V(P)$;
        \item $W^\bot$ is both a $\preceq$-upset and a $\sqsubseteq$-upset.
    \end{itemize}
    Backward confluence is illustrated in Figure \ref{figure::confluences}.
\end{definition}

Fix a $\csf$-model $M = \tuple{W, W^\bot,\preceq, \R,V}$.
Define the valuation of the formulas over $M$ by induction on the structure of the formulas:
\begin{itemize}
    \item $M,w\models P$ iff  $w\in V(P)$;
    \item $M,w\models \bot$ iff $w\in W^\bot$;
    \item $M,w\models \varphi\land\psi$ iff $M,w\models\varphi$ and $M,w\models\psi$;
    \item $M,w\models \varphi\lor\psi$ iff $M,w\models\varphi$ or $M,w\models\psi$;
    \item $M,w\models \varphi\to\psi$ iff for all $v\in W$, if $w\preceq v$ and $M,v\models\varphi$, then $M,v\models\psi$;
    \item $M,w\models \Box\varphi$ iff for all $v,u\in W$, if $w\preceq v$ and $v\R u$, then $M,u\models\varphi$; and
    \item $M,w\models \Diamond\varphi$ iff for all $v\in W$, if $w\preceq v$ then there is $u$ such that $v \R u$ and $M,u\models\varphi$.
\end{itemize}
Given a $\csf$-model $M$ and a formula $\varphi$, we define \[\|\varphi\|^M := \{w\in W \mid M,w\models\varphi\}.\]

\begin{lemma}
    \label{lem::preservation}
    Let $M = \tuple{W, W^\bot,\preceq, \R,V}$ be a $\csf$-model and $\varphi$ be a formula.
    Then $\|\varphi\|^M$ is a $\preceq$-upset.
\end{lemma}
 
Models for the other logics we consider are obtained by enforcing various `confluence' properties.
\begin{definition}
    \label{def::birelational-properties}
    We say that a $\csf$-model $M = \tuple{W, W^\bot, \preceq, \R, V}$ is:
    \begin{itemize}
        \item \emph{forward confluent} iff $w \R v$ and $w\preceq w'$ implies there is $v'$ such that $w' \R v'$ and $v\preceq v'$;
        \item \emph{downward confluent} iff $w \preceq  v \R v'$ implies there is $w'$ such that $w \R w'$ and $w' \preceq v'$;
        \item \emph{locally linear} iff $w\preceq v$ and $w\preceq u$ implies that either $v\preceq u$ or $u\preceq v$.
    \end{itemize}
    Forward and downward confluence are illustrated in Figure \ref{figure::confluences}.
    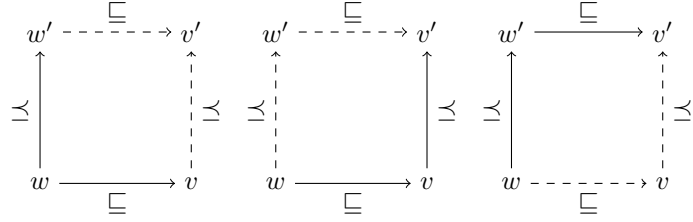
\begin{figure}[ht]
        \centering
        \tikzstyle{world}=[circle,draw,minimum size=5mm,inner sep=0pt]
        \begin{tikzpicture}
            \node (w) at (0,-2) {$w$};
            \node (w2) at (0,0) {$w'$};
            \node (v) at (2,-2) {$v$};
            \node (v2) at (2,0) {$v'$};
        
            \draw[->] (w) -- (w2) node[midway,left] {$\preceq$};
            \draw[->] (w) -- (v) node[midway,below] {$\R$};
            \draw[dashed,->] (v) -- (v2) node[midway,right] {$\preceq$};
            \draw[dashed,->] (w2) -- (v2) node[midway,above] {$\R$};
        \end{tikzpicture}
        \begin{tikzpicture}
            \node (w) at (0,-2) {$w$};
            \node (w2) at (0,0) {$w'$};
            \node (v) at (2,-2) {$v$};
            \node (v2) at (2,0) {$v'$};
        
            \draw[dashed,->] (w) -- (w2) node[midway,left] {$\preceq$};
            \draw[->] (w) -- (v) node[midway,below] {$\R$};
            \draw[->] (v) -- (v2) node[midway,right] {$\preceq$};
            \draw[dashed,->] (w2) -- (v2) node[midway,above] {$\R$};
        \end{tikzpicture}
        \begin{tikzpicture}
            \node (w) at (0,-2) {$w$};
            \node (w2) at (0,0) {$w'$};
            \node (v) at (2,-2) {$v$};
            \node (v2) at (2,0) {$v'$};
        
            \draw[->] (w) -- (w2) node[midway,left] {$\preceq$};
            \draw[dashed,->] (w) -- (v) node[midway,below] {$\R$};
            \draw[dashed,->] (v) -- (v2) node[midway,right] {$\preceq$};
            \draw[->] (w2) -- (v2) node[midway,above] {$\R$};
        \end{tikzpicture}
        \caption{Schematics for forward (left), backward (center), and downward (right) confluence. Solid arrows correspond to the universal quantifiers and dashed arrows correspond to the existential quantifiers.}
        \label{figure::confluences}
    \end{figure}
\end{definition}

\begin{lemma}
    \label{lem::classical-valuations}
    Let $M = \tuple{W, W^\bot, \preceq, \R, V}$ be a $\csf$-model.
    Then:
    \begin{itemize}
        \item if $M$ is forward confluent, then $M, w\models \Diamond\varphi$ iff there is $v\in W$ such that $w \R v$ and $M,v\models\varphi$;
        \item if $M$ is downward confluent, then $M,w\models\Box\varphi$ iff for all $v\in W$, if $w \R v$, then $M,v\models\varphi$.
    \end{itemize}
\end{lemma}

\begin{definition}
    Let $M = \tuple{W, W^\bot, \preceq, \R, V}$ be a birelational $\csf$-model.
    Then, $M$ is:
    \begin{itemize}
        \item an \emph{$\isf$-model} iff it is forward and backward confluent;
        \item a \emph{$\gsf$-model} iff it is a locally linear $\isf$-model; and
        \item a \emph{$\gsfc$-model} iff it is a downward confluent $\gsf$-model.
    \end{itemize}
\end{definition}

Soundness and completeness for the four logics we consider are known~\cite{alechina2001cs4,balbiani2024variants,servi1984axiomatizations,simpson1994phdthesis}:
\begin{theorem}
    \label{thm::completeness-nonclassical}
    Let $\Lambda \in \{\csf, \isf, \gsf, \gsfc\}$. 
    Then $\Lambda$ is sound and strongly complete with respect to the class of birelational $\Lambda$-models.
\end{theorem}

\paragraph{Modalities.}
Let $X\subseteq\{\neg,\Box,\Diamond\}$.
An \emph{$X$-modality} is a finite sequence $\sigma \in X^{<\omega}$.
We denote the empty $X$-modality by $\varepsilon$.
Fix $X$-modalities $\sigma$ and $\tau$.
We denote the concatenation of $\sigma$ and $\tau$ by $\sigma\tau$.
We define inductively $\sigma^0 :=\varepsilon$ and $\sigma^{n+1} := \sigma\sigma^n$.
If $\varphi$ is a formula, $\sigma\varphi$ denotes the formula obtained by appending the modality $\sigma$ to $\varphi$.

We say $\sigma$ and $\tau$ are \emph{$\Lambda$-equivalent} if, for all $\Lambda$-model $M$ and for all formula $\varphi$, $\|\sigma\varphi\|^M = \|\tau\varphi\|^M$.
We say $\Lambda$ has \emph{finitely many $X$-modalities} iff there is a finite set of modalities $\Sigma\subseteq X^{<\omega}$ such that every $X$-modality $\tau$ is $\Lambda$-equivalent to some $\sigma\in \Sigma$.
We say $\Lambda$ has \emph{infinitely many $X$-modalities} iff $\Lambda$ does not have finitely many $X$-modalities.

\begin{proposition}
    \label{lem::transfer}
    Let $\Lambda,\Lambda'\in\{\csf,\isf,\gsf,\gsfc\}$ be such that $\Lambda\subseteq \Lambda'$ and fix $X\subseteq\{\neg,\Box,\Diamond\}$.
    If $\Lambda$ has finitely many $X$-modalities, then $\Lambda'$ also has finitely many $X$-modalities.
    Conversely, if $\Lambda'$ has infinitely many $X$-modalities, then $\Lambda$ also has infinitely many $X$-modalities.
\end{proposition}

%%%%%%%%%%%%%%%%%%%%%%%%%%%%%%%%%%%%%%%%%%%%%%%%%%%%%%%%%%%%%%%%%%%%%%%%%%%%%%%%
\section{The case for \texorpdfstring{$\csf$}{CS4}}
%%%%%%%%%%%%%%%%%%%%%%%%%%%%%%%%%%%%%%%%%%%%%%%%%%%%%%%%%%%%%%%%%%%%%%%%%%%%%%%%
We first show that many reductions involving negations and boxes are available in $\csf$:
\begin{lemma}
    \label{lem::reductions-csf}
    For all formula $\varphi$:
    \begin{enumerate}
        \item \label{it::nnn} 
            $\csf\vdash \neg\neg\neg\varphi\leftrightarrow \neg \varphi$;
        \item \label{it::bb} 
            $\csf\vdash \Box\Box\varphi \leftrightarrow \Box\varphi$;
        \item \label{it::dd} 
            $\csf\vdash \Diamond\Diamond \varphi \leftrightarrow\Diamond\varphi$;
        \item \label{it::bdbd} 
            $\csf\vdash \Box\Diamond\Box\Diamond \varphi \leftrightarrow \Box\Diamond\varphi$;
        \item \label{it::dbdb} 
            $\csf\vdash \Diamond\Box\Diamond\Box \varphi \leftrightarrow \Diamond\Box \varphi$;
        \item \label{it::ndd} 
            $\csf\vdash \Box\neg \neg\Box\neg  \varphi \leftrightarrow \Box\neg \varphi$; and
        \item \label{it::nbdbd} 
            $\csf\vdash \Box\neg\Box\neg\Box\neg\Box\neg \varphi \leftrightarrow \Box\neg\Box\neg\varphi$.
    \end{enumerate}
\end{lemma}
\begin{proof}
    Item \ref{it::nnn} is an intuitionistic tautology.
    Items \ref{it::bb}, \ref{it::dd}, \ref{it::bdbd}, and \ref{it::dbdb} follow from axioms $\ax{4}$ and $\ax{T}$ as in the classical case.

    We now prove Item \ref{it::ndd}.
    As $\Box\neg\varphi \to \neg\neg\Box\neg\varphi$ is an intuitionistic tautology, we have $\csf\vdash\Box\Box\neg\varphi \to \Box\neg\neg\Box\varphi$ by $\ax{Nec}$ and $\ax{K}$. By $\ax{4}$, we have $\csf\vdash \Box\neg\varphi \to \Box\neg\neg\Box\neg\varphi$.
    On the other hand, $\Box\neg\varphi \to \neg\varphi$ is an instance of $\ax{T}$.
    By combining intuitionistic tautologies with the use of $\ax{MP}$, we get $\csf\vdash\varphi \to \neg\Box\neg\varphi$. 
    By contraposition, we get $\csf\vdash \neg\neg\Box\neg\varphi \to \neg\varphi$.
    At last, we get $\csf\vdash \Box\neg\neg\Box\neg\varphi \to \Box\neg\varphi$ by $\ax{Nec}$ and $\ax{K}$.
    Therefore $\csf\vdash \Box\neg \neg\Box\neg  \varphi \leftrightarrow \Box\neg \varphi$.
    
    We now prove Item \ref{it::nbdbd}.
    First note that $\Box\neg\Box\neg\Box\neg\varphi\to \neg\Box\neg\Box\neg\varphi$ is an instance of $\ax{T}$.
    By contraposition, we get $\csf\vdash\neg\neg\Box\neg\Box\neg\varphi \to \neg\Box\neg\Box\neg\Box\neg\varphi$.
    By $\ax{Nec}$ and $\ax{K}$, we have $\csf\vdash\Box\neg\neg\Box\neg\Box\neg\varphi \to \Box\neg\Box\neg\Box\neg\Box\neg\varphi$.
    By Item \ref{it::ndd}, we have $\csf\vdash\Box\neg\Box\neg\varphi \to \Box\neg\Box\neg\Box\neg\Box\neg\varphi$.
    On the other hand, note that $\Box\neg\Box\neg \varphi \to \neg\Box\neg\varphi$ is an instance of $\ax{T}$.
    By contraposition and then by applications of $\ax{Nec}$ and $\ax{K}$, we get $\csf\vdash \Box\neg\neg\Box\neg\varphi \to \Box\neg\Box\neg\Box\neg \varphi$.
    By Item \ref{it::ndd}, $\csf\vdash \Box\neg\varphi \to \Box\neg\Box\neg\Box\neg \varphi$.
    By another use of contraposition, $\ax{Nec}$, and $\ax{K}$, we get $\csf\vdash \Box\neg\Box\neg\Box\neg\Box\neg \varphi\to \Box\neg \Box\neg \varphi$.
    Therefore $\csf\vdash \Box\neg\Box\neg\Box\neg\Box\neg \varphi \leftrightarrow \Box\neg\Box\neg\varphi$.
\end{proof}

From Lemma \ref{lem::reductions-csf}, we obtain:
\begin{theorem}
    $\csf$ has finitely many $\{\Box,\Diamond\}$- and $\{\neg,\Box\}$-modalities.
\end{theorem}

We now define a $\csf$-model which will witness that $\csf$ has infinitely many $\{\neg,\Diamond\}$-modalities.

\begin{definition}
    \label{def::csf-witness}
    Let $M_\infty^\csf = \tuple{W, W^\bot, \preceq, \R, V}$ be defined as follows.
    Set
    \begin{align*}
        W := &\{a_i^{\emptyset} \mid i\in\omega \} \cup \{a_i^{\{i'\}} \mid i,i'\in\omega \text{ and } i\neq i' \} \\
        &\cup \{b_i^{\emptyset} \mid i\in\omega\} 
         \cup \{b_i^{\{i-1\}} \mid i\in\omega \text{ and } i\geq 1\},
    \end{align*}
    and $W^\bot := \emptyset$.
    We have $w\preceq v$ iff one of the following holds:
    \begin{itemize}
        \item $w = a_i^s$ and $v = a_i^{s'}$;
        \item $w = a_i^s$ and $v = b_i^{s'}$; or
        \item $w = b_i^s$ and $v = b_i^{s'}$.
    \end{itemize}
    Set $\ell(a_i^s) := s$ and $\ell(b_i^{s'}) := s'$, we have $w\sqsubseteq v$ iff $\ell(w) \subseteq \ell(v)$.
    At last, set 
    \[
        V(P) :=
        \{a_0^{\emptyset},b_0^{\emptyset} \} \cup \{a_0^{\{i'\}} \mid i'\in\omega \text{ and } i'\neq 0 \}.
    \]
    The model $M_\infty^\csf$ is illustrated in Figure \ref{figure::csf-witness-model}.
    
    \begin{figure}[ht]
    \centering
    \usetikzlibrary{fit}
    % \tikzstyle{world}=[circle,draw,minimum size=5mm,inner sep=0pt]
    \tikzstyle{world}=[]
    \begin{tikzpicture}
        \node[world] (a0e)  at (0,0) {$a_0^{\emptyset}$};
        % \node[world] (a00)  at (0,1) {$a_0^{\{0\}}$};
        \node[world] (a01)  at (0,2) {$a_0^{\{1\}}$};
        \node[world] (a02)  at (0,3) {$a_0^{\{2\}}$};
        \node[world] (a03)  at (0,4) {$a_0^{\{3\}}$};
        \node[world] (a0d)  at (0,5) {$\vdots$};
        % \node[world] (a0l)  at (0,6) {\rotatebox[origin=c]{90}{$\preceq$}};
        \node[world] (b0e)  at (0,7) {$b_0^{\emptyset}$};
        % \node[world] (b0m)  at (0,8) {$b_0^{\{-1\}}$};
        
        \node[world] (a1e)  at (1.5,0) {$a_1^{\emptyset}$};
        \node[world] (a10)  at (1.5,1) {$a_1^{\{0\}}$};
        % \node[world] (a11)  at (1.5,2) {$a_1^{\{1\}}$};
        \node[world] (a12)  at (1.5,3) {$a_1^{\{2\}}$};
        \node[world] (a13)  at (1.5,4) {$a_1^{\{3\}}$};
        \node[world] (a1d)  at (1.5,5) {$\vdots$};
        \node[world] (b1e)  at (1.5,7) {$b_1^{\emptyset}$};
        \node[world] (b1m)  at (1.5,8) {$b_1^{\{0\}}$};

        \node[world] (a2e)  at (3,0) {$a_2^{\emptyset}$};
        \node[world] (a20)  at (3,1) {$a_2^{\{0\}}$};
        \node[world] (a21)  at (3,2) {$a_2^{\{1\}}$};
        % \node[world] (a22)  at (3,3) {$a_2^{\{2\}}$};
        \node[world] (a23)  at (3,4) {$a_2^{\{3\}}$};
        \node[world] (a2d)  at (3,5) {$\vdots$};
        \node[world] (b2e)  at (3,7) {$b_2^{\emptyset}$};
        \node[world] (b2m)  at (3,8) {$b_2^{\{1\}}$};

        \node[world] (a3e)  at (4.5,0) {$a_3^{\emptyset}$};
        \node[world] (a30)  at (4.5,1) {$a_3^{\{0\}}$};
        \node[world] (a31)  at (4.5,2) {$a_3^{\{1\}}$};
        \node[world] (a32)  at (4.5,3) {$a_3^{\{2\}}$};
        % \node[world] (a33)  at (4.5,4) {$a_3^{\{3\}}$};
        \node[world] (a3d)  at (4.5,5) {$\vdots$};
        \node[world] (b3e)  at (4.5,7) {$b_3^{\emptyset}$};
        \node[world] (b3m)  at (4.5,8) {$b_3^{\{2\}}$};
        
        \node[world] (a4e)  at (6,0) {$\cdots$};
        \node[world] (a40)  at (6,1) {$\dots$};
        \node[world] (a41)  at (6,2) {$\dots$};
        \node[world] (a42)  at (6,3) {$\dots$};
        \node[world] (a43)  at (6,4) {$\dots$};
        % \node[world] (a4d)  at (6,5) {$\vdots$};
        \node[world] (b4e)  at (6,7) {$\dots$};
        \node[world] (b4m)  at (6,8) {$\dots$};

        \node[draw,fit=(a0e) (a0d)] {};
        \node[draw,fit=(a1e) (a1d)] {};
        \node[draw,fit=(a2e) (a2d)] {};
        \node[draw,fit=(a3e) (a3d)] {};
        
        \node[draw,fit=(b0e)      ] {};
        \node[draw,fit=(b1e) (b1m)] {};
        \node[draw,fit=(b2e) (b2m)] {};
        \node[draw,fit=(b3e) (b3m)] {};

        \draw[->] (a0d) -- (b0e) node[midway,left] {$\preceq$};
        \draw[->] (a1d) -- (b1e) node[midway,left] {$\preceq$};
        \draw[->] (a2d) -- (b2e) node[midway,left] {$\preceq$};
        \draw[->] (a3d) -- (b3e) node[midway,left] {$\preceq$};
    \end{tikzpicture}
    \caption{The model $M_\infty^\csf$ witnessing the pairwise non-equivalence of $\{(\neg\Diamond)^n P \mid n\in\omega\}$ over $\csf$. We denote only the $\preceq$-clusters.}
    \label{figure::csf-witness-model}
    \end{figure}
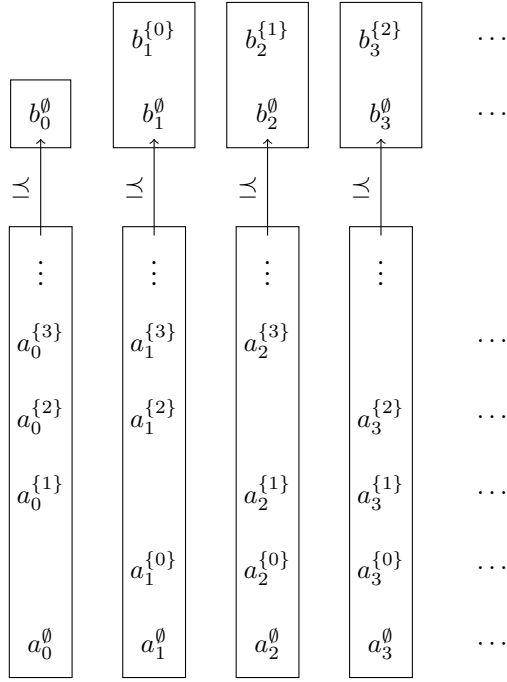
\end{definition}

\begin{lemma}
    \label{lem::is-csf-model}
    Let $M_\infty^\csf = \tuple{W, W^\bot, \preceq, \R, V}$ be defined as above.
    Then $M_\infty^\csf$ is a $\csf$-model.
\end{lemma}
\begin{proof}
    It is immediate from the definition that $\preceq$ and $\R$ are preorders and that $V(P)$ is a $\preceq$-upset.
    As $W^\bot = \emptyset$, its requirements are vacuously met.

    The only interesting property is backward confluence.
    Suppose that $w\R v \preceq v'$.
    If $w = a_i^s$ or $w = b_i^s$, let $w'$ be $b_i^{\{\emptyset\}}$.
    Then $w \preceq w'$ and $w' \R v'$.
    Therefore $\mcsf$ is backward confluent.
\end{proof}

\begin{lemma}
    \label{lem::csf-witnesses-infinity}
    Let $\mcsf = \tuple{W, W^\bot, \preceq, \R, V}$ be defined as above.
    For all $n\in\omega$:
    \[
        \|(\neg\Diamond)^{n} P\|^\mcsf = \{w \in W \mid \ell(w) = n+1 \}.
    \]
\end{lemma}
\begin{proof}
    We prove this lemma by induction on $n$.
    The case for $n = 0$ follows directly from the definition of $\mcsf$, as $\| P \|^\mcsf = \{a_0^{\emptyset},b_0^{\emptyset} \} \cup \{a_0^{\{i'\}} \mid i'\in\omega \text{ and } i'\neq 0 \}$.

    Suppose that $\|(\neg\Diamond)^{n} P\|^\mcsf = \{a_i^s, b_i^s \in W \mid i=n \}$.
    First note that $w\preceq b_{n+1}^{n}$ holds for any $w\in W$ such that $\ell(w) = n+1$; as $b_{n+1}^s$ does not access by $\R$ any $v$ such that $\ell(v) = n$, we have that $\{w \in W \mid \ell(w) = n+1 \} \subseteq \|(\neg\Diamond)^{n} P\|^\mcsf$.
    Now, note that if $m\neq n+1$, then $M,b_m^\emptyset \models \Diamond(\neg\Diamond)^n P$: we have that $b_m^\emptyset \R b_n^\emptyset$ and, if $m\neq 0$, $b^{m-1}_m \R a_n^{m-1}$ since $m-1\neq n$.
    As $w\preceq b_m^\emptyset$ for any $w\in W$ such that $\ell(w) = m$, we have $\|(\neg\Diamond)^{n} P\|^\mcsf \subseteq \{w \in W \mid \ell(w) = n+1 \}$.
    Therefore $\|(\neg\Diamond)^{n} P\|^\mcsf = \{w \in W \mid \ell(w) = n+1 \}$.
\end{proof}

Lemma \ref{lem::csf-witnesses-infinity} readily implies that $\{(\neg\Diamond)^n \in\{\neg,\Diamond\}^{<\omega} \mid n\in\omega\}$ is a set of pairwise non-equivalent modalities. Therefore:
\begin{theorem}
    $\csf$ has infinitely many $\{\neg,\Diamond\}$-modalities.
\end{theorem}

%%%%%%%%%%%%%%%%%%%%%%%%%%%%%%%%%%%%%%%%%%%%%%%%%%%%%%%%%%%%%%%%%%%%%%%%%%%%%%%%
\section{The case for \texorpdfstring{$\isf$}{IS4} and \texorpdfstring{$\gsf$}{GS4}}
%%%%%%%%%%%%%%%%%%%%%%%%%%%%%%%%%%%%%%%%%%%%%%%%%%%%%%%%%%%%%%%%%%%%%%%%%%%%%%%%
We first show that many reductions involving negations and boxes are available in $\isf$:
\begin{lemma}
    \label{lem::reductions-isf}
    Let $\varphi$ be a formula. Then:
    \begin{enumerate}
        \item \label{it::swap}
            $\isf \vdash \Box\neg\varphi \leftrightarrow \neg\Diamond\varphi$;
        \item \label{it::nbb} 
            $\isf \vdash \neg\Diamond\neg \neg\Diamond\varphi \leftrightarrow \neg\Diamond \varphi$; and
        \item \label{it::ndbdb} 
            $\isf \vdash \neg\Diamond\neg\Diamond\neg\Diamond\neg\Diamond \varphi \leftrightarrow \neg\Diamond\neg\Diamond\varphi$.
    \end{enumerate}
\end{lemma}
\begin{proof}
    Fix a formula $\varphi$.
    By $\ax{K}$, we have $\isf\vdash \Box(\varphi\to\bot)\to(\Diamond\varphi\to\Diamond\bot)$.
    By $\ax{N_\Diamond}$, we have  $\isf\vdash \Box(\varphi\to\bot)\to(\Diamond\varphi\to\bot)$.
    That is, $\isf\vdash \Box\neg\varphi \to \neg\Diamond\varphi$.
    On the other hand, $\bot\to\Box\bot$ is an intuitionistic tautology.
    So, by propositional reasoning, we have $\isf \vdash (\Diamond\varphi \to \bot) \to (\Diamond\varphi \to \Box\bot)$.
    As $(\Diamond\varphi\to\Box\bot) \to \Box(\varphi\to\bot)$ is an instance of $\ax{FS}$, we have $\isf\vdash (\Diamond\varphi\to\bot)\to \Box(\varphi\to\bot)$.
    That is, $\isf\vdash \neg\Diamond\varphi\to\Box\neg\varphi$.
    This concludes the proof of Item \ref{it::swap}.

    Items \ref{it::nbb} and \ref{it::ndbdb} then follow from Item \ref{it::swap} along with Lemma \ref{lem::reductions-csf}.
\end{proof}

From Lemma \ref{lem::reductions-isf}, we obtain:
\begin{theorem}
    $\isf$ has finitely many $\{\neg,\Diamond\}$-modalities.
\end{theorem}

We now define a $\gsf$-model which will witness that $\gsf$ has infinitely many $\{\neg,\Diamond\}$-modalities.
\begin{definition}
    \label{def::gsfc-witness}
    Let $\mgsf = \tuple{W, W^\bot, \preceq, \R, V}$ be defined as follows.
    Set $W := \{a\} \cup \{b_i \mid i\in\omega\} \cup \{c_i \mid i\in\omega\}$ and $W^\bot := \emptyset$.
    We have $w\preceq v$ iff one of the following holds:
    \begin{itemize}
        \item $w = v$; or
        \item $w = c_i$ and $v = b_i$.
    \end{itemize}
    We have $w\sqsubseteq v$ iff one of the following holds:
    \begin{itemize}
        \item $w = v$;
        \item $w = b_i$ and $v = a$;
        \item $w = b_i$, $v = b_j$, and $j\leq i$; or
        \item $w = b_i$, $v = c_j$, and $j < i$.
    \end{itemize}
    At last, set $V(P) := \{w\}$.
    
    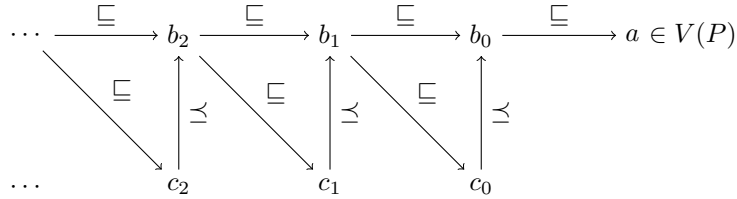
\begin{figure}[ht]
    \centering
    % \tikzstyle{world}=[circle,draw,minimum size=5mm,inner sep=0pt]
    \tikzstyle{world}=[]
    \begin{tikzpicture}
        \node[world] (a)  at (6,0) {$a$};
        \node        (p)  at (6.8,0) {$\in V(P)$};
        
        \node[world] (c0) at (4,-2) {$c_0$};
        \node[world] (b0) at (4, 0) {$b_0$};
        
        \node[world] (c1) at (2,-2) {$c_1$};
        \node[world] (b1) at (2, 0) {$b_1$};
        
        \node[world] (c2) at (0,-2) {$c_2$};
        \node[world] (b2) at (0, 0) {$b_2$};
        
        \node        (c3) at (-2,-2) {$\cdots$};
        \node        (b3) at (-2, 0) {$\cdots$};

        \draw[->] (b0) -- (a) node[midway,above] {$\R$};
        
        \draw[->] (c0) -- (b0) node[midway,right] {$\preceq$};
        \draw[->] (c1) -- (b1) node[midway,right] {$\preceq$};
        \draw[->] (c2) -- (b2) node[midway,right] {$\preceq$};
        
        \draw[->] (b1) -- (b0) node[midway,above] {$\R$};
        \draw[->] (b2) -- (b1) node[midway,above] {$\R$};
        \draw[->] (b3) -- (b2) node[midway,above] {$\R$};
        
        \draw[->] (b1) -- (c0) node[midway,above right] {$\R$};
        \draw[->] (b2) -- (c1) node[midway,above right] {$\R$};
        \draw[->] (b3) -- (c2) node[midway,above right] {$\R$};
    \end{tikzpicture}
    \caption{The model $\mgsf$ witnessing the pairwise non-equivalence of $\{(\Diamond\neg\neg\Box)^n P \mid n\in\omega\}$ over $\gsf$. We omit reflexive and transitive arrows.}
    \label{figure::gsfc-witness-model}
    \end{figure}
\end{definition}

\begin{lemma}
    \label{lem::is-gsf-model}
    Let $M_\infty^\gsf = \tuple{W, W^\bot, \preceq, \R, V}$ be defined as above.
    Then $M_\infty^\gsf$ is a $\gsf$-model.
\end{lemma}
\begin{proof}
    It is straightforward to see that $\preceq$ and $\R$ are preorders over $W$ and that $V(P)$ is a $\preceq$-upset.
    As there are no fallible worlds, the conditions on $W^\bot$ are vacuously true.

    Suppose $w \R v \preceq v'$.
    Then either $w = a$ or $w = b_i$.
    In case $w = a$, then $v = v' = a$ too by definition, and so $w \preceq w \R w$.
    In case $w = b_i$, we have two possibilities for $v$.
    If $v = a$, then $v' = a$ too, so $w\preceq w \R a$.
    If $v = b_j$ or $v = c_j$ then $w \R b_j$ too, so $w\preceq w \R b_j$.
    
    Suppose $w\preceq w'$ and $w\R v$.
    If $w = v$, we may set $v' := w'$ and we are done.
    If $w\neq v$, then $w = b_i$ for some $i\in\omega$.
    So $w' = b_i$ too, and so by setting $v' := v$ we are done.

    Now, suppose that $w\preceq v$ and $w\preceq u$.
    Then there is $n\in\omega$ such that $n = \ell(w) = \ell(v) = \ell(u)$.
    If $v = b^s_n$, then $u\preceq v$.
    Similarly, $u = b^s_n$, then $v\preceq u$.
    Otherwise, $v = a_n^s$ and $u = a_n^{s'}$, and so $v\preceq u$.
\end{proof}

\begin{lemma}
    \label{lem::gsf-witnesses-infinity}
    Let $M_\infty^\gsf = \tuple{W, W^\bot, \preceq, \R, V}$ be defined as above.
    Then:
    \[
        \|(\Diamond\neg\neg\Box)^{n+1} P\|^\mgsf = \{a\} \cup \{b_i\mid i\in\omega\} \cup \{ c_i \mid i<n \}.
    \]
\end{lemma}
\begin{proof}
    This proof is by induction on $n$.

    By definition, we have $\|P\|^\mgsf = \|\Box P\|^\mgsf = \|\neg\neg\Box P\|^\mgsf = \{a\}$.
    Therefore $\mgsf,w\models \Diamond\neg\neg\Box P$ iff $w = a$ or $w = b_i$ for some $i\in\omega$.

    Now, suppose $\|(\Diamond\neg\neg\Box)^{n+1} P\|^\mgsf = \{a\} \cup \{b_i\mid i\in\omega\} \cup \{ c_i \mid i<n \}$.
    If $w\in \{a\} \cup \{b_i\mid i\in\omega\} \cup \{ c_i \mid i<n \}$ the induction hypothesis along with a straightforward computation shows $\mgsf, w\models\Diamond\neg\neg\Box\varphi$.
    Now note that $\mgsf, b_{n}\models\Box(\Diamond\neg\neg\Box)^{n+1} P$, so $\mgsf, c_{n}\models\neg\neg\Box(\Diamond\neg\neg\Box)^{n+1} P$, and thus $\mgsf, c_{n}\models\Diamond\neg\neg\Box(\Diamond\neg\neg\Box)^{n+1} P$ as $\R$ is reflexive.
    Therefore $\{a\} \cup \{b_i\mid i\in\omega\} \cup \{ c_i \mid i<n \} \subseteq \|(\Diamond\neg\neg\Box)^{n+2} P\|^\mgsf$.
    Now, if $w = c_m$ with $m\geq n+1$, then $\mgsf, b_m\not\models\Box(\Diamond\neg\neg\Box)^{n+1} P$, and so $\mgsf, c_m \not\models \neg\neg\Box(\Diamond\neg\neg\Box)^{n+1} P$; as the only world $c_m$ accesses via $\R$ is itself, $\mgsf, c_m \not\models \Diamond\neg\neg\Box(\Diamond\neg\neg\Box)^{n+1} P$.
    We conclude $\|(\Diamond\neg\neg\Box)^{n+2} P\|^M = \{a\} \cup \{b_i\mid i\in\omega\} \cup \{ c_i \mid i<n+1 \}$.
\end{proof}

Lemma \ref{lem::gsf-witnesses-infinity} readily implies that $\{(\Diamond\neg\neg\Box)^{n+1} \in\{\neg,\Box,\Diamond\}^{<\omega} \mid n\in\omega\}$ is a set of pairwise non-equivalent modalities. Therefore:
\begin{theorem}
    \label{thm::gsf-has-inf-many-modalities}
    $\gsf$ has infinitely many $\{\neg,\Box,\Diamond\}$-modalities.
\end{theorem}

%%%%%%%%%%%%%%%%%%%%%%%%%%%%%%%%%%%%%%%%%%%%%%%%%%%%%%%%%%%%%%%%%%%%%%%%%%%%%%%%
\section{The case for \texorpdfstring{$\gsfc$}{GS4c}}
%%%%%%%%%%%%%%%%%%%%%%%%%%%%%%%%%%%%%%%%%%%%%%%%%%%%%%%%%%%%%%%%%%%%%%%%%%%%%%%%
To prove that $\gsfc$ has finitely many modalities, we show that, under the right circumstances, boxes and diamonds are essentially dual. 

\begin{lemma}
    \label{lem::updownset-negation}
    Let $M = \tuple{W, W^\bot, \preceq, \R, V}$ be a $\gsf$-model.
    For all $\varphi$, $\|\neg\varphi\|^M$ is a $\preceq$-updownset.
\end{lemma}
\begin{proof}
    By Lemma \ref{lem::preservation}, $\|\neg\varphi\|^M$ is a $\preceq$-upset.
    Now suppose $w\in \|\neg\varphi\|^M$ and $v\preceq w$.
    If $M,v\models\neg\varphi$, we have nothing to do.
    Otherwise, there is $u\in W$ such that $v\preceq u $ and $M,u\models\varphi$.
    By the local linearity of $\preceq$, $w\preceq u$ or $u\preceq w$.
    In the first case, $M,w\models\neg\varphi$ implies $M,u\models\neg\varphi$, a contradiction.
    In the second case, $M,u\models\varphi$ implies $M,w\models\varphi$, again a contradiction.
    Therefore we must have $v\in\|\neg\varphi\|^M$.
\end{proof}

\begin{lemma}
    \label{lem::updownset-box-diamond}
    Let $M = \tuple{W, W^\bot, \preceq, \R, V}$ be a $\gsfc$-model.
    If $\|\varphi\|$ is a $\preceq$-updownset then $\|\Box\varphi\|$ and $\|\Diamond\varphi\|$ are also $\preceq$-updownsets.
\end{lemma}
\begin{proof}
    By Lemma \ref{lem::preservation}, $\|\Box\varphi\|^M$ and $\|\Diamond\varphi\|^M$ are $\preceq$-upsets.

    Suppose that $w\in \|\Box\varphi\|^M$ and $v\preceq w$.
    Let $v'$ be such that $v \R v'$.
    By forward confluence, there is $w'$ such that $w \R w'$ and $v'\preceq w'$.
    As $M,w\models\Box\varphi$, we have $M,w'\models\varphi$.
    As $\|\varphi\|^M$ is a $\preceq$-updownset, $M,v'\models\varphi$ too.
    Therefore $M,v\models\Box\varphi$; that is, $v\in \|\Box\varphi\|^M$.
    
    Suppose that $w\in \|\Diamond\varphi\|^M$ and $v\preceq w$.
    There is $w'$ such that $w \R w'$ and $M,w'\models\varphi$.
    By downward confluence, there is $v'$ such that $v \R v'$ and $v'\preceq w'$.
    As $\|\varphi\|^M$ is a $\preceq$-updownset, $M,v'\models\varphi$ too.
    Therefore $M,v\models\Diamond\varphi$; that is, $v\in \|\Diamond\varphi\|^M$.
\end{proof}

\begin{lemma}
    \label{lem::flat-modalities}
    Let $M = \tuple{W, W^\bot, \preceq, \R, V}$ be a $\gsfc$-model.
    If $\|\varphi\|$ is a $\preceq$-updownset, then
    \[\text{
        $\|\neg\neg\varphi\|^M = \|\varphi\|^M$,
        $\|\Box\neg\varphi\|^M = \|\neg\Diamond\varphi\|^M$, and
        $\|\Diamond\neg\varphi\|^M = \|\neg\Box\varphi\|^M$.}  
    \]
\end{lemma}
\begin{proof}
    Fix a $\gsfc$-model $M = \tuple{W, W^\bot, \preceq, \R, V}$, a world $w\in W$, and a formula $\varphi$ such that $\|\varphi\|^M$ is a $\preceq$-updownset.

    If $M,w\models\varphi$, then $M,w\models \neg\neg\varphi$ readily follows.
    Suppose that $M,w\models \neg\neg\varphi$.
    So there is $v$ such that $w\preceq v$ and $M,v\models \varphi$.
    As $\|\varphi\|^M$ is a $\preceq$-updownset, $M,w\models\varphi$ too.
    We conclude $\|\neg\neg\varphi\|^M = \|\varphi\|^M$.

    Suppose that $M,w\models \Box\neg\varphi$.
    Let $w'$ be such that $w\preceq w'$.
    If $v'$ is such that $w' \R v'$, then $M,v'\models\neg\varphi$.
    Therefore $M,w'\not\models\Diamond\varphi$, and so $M,w\models\neg\Diamond\varphi$.
    Now, suppose $M,w\models\neg\Diamond\varphi$.
    Let $v$ be such that $w\R v$.
    Then $M,w\not\models\Diamond\varphi$, and so $M, v\not\models \varphi$.
    As $\|\varphi\|^M$ is a $\preceq$-updownset, $M,v\models\neg\varphi$.
    Therefore $M,w\models\Box\neg\varphi$.
    We conclude $\|\Box\neg\varphi\|^M = \|\neg\Diamond\varphi\|^M$.

    Suppose that $M,w\models \Diamond\neg\varphi$.
    Let $w'\in W$ be such that $w\preceq w'$.
    Then there is $v'\in W$ such that $w'\R v'$ and $M,v'\models\neg\varphi$.
    Therefore $M, w'\not\models\Box\varphi$.
    And so $M,w\models\neg\Box\varphi$.
    Now suppose $M,w\models\neg\Box\varphi$.
    As $M,w\not\models\Box\varphi$, there is $v\in W$ such that $w\R v$ and $M,v\not\models\varphi$.
    As $\|\varphi\|^M$ is a $\preceq$-updownset, $M,v\models\neg\varphi$.
    So $M,w\models\Diamond\neg\varphi$.
    We conclude $\|\Diamond\neg\varphi\|^M = \|\neg\Box\varphi\|^M$.
\end{proof}

Finally, we can show:
\begin{theorem}
    \label{thm::gsfc-has-fin-many-modalities}
    $\gsfc$ has finitely many $\{\neg,\Box,\Diamond\}$-modalities.
\end{theorem}
\begin{proof}
    Let $M = \tuple{W, W^\bot, \preceq, \R, V}$ be a $\gsfc$-model.
    By Lemma \ref{lem::reductions-csf}, there are only $7$ $\{\Box,\Diamond\}$-modalities: $\varepsilon$, $\Box$, $\Diamond$, $\Box\Diamond$, $\Diamond\Box$, $\Box\Diamond\Box$, and $\Diamond\Box\Diamond$.
    Let $\sigma$ be a $\{\neg,\Box,\Diamond\}$-modality where $\neg$ occurs.
    Write $\sigma = \sigma_l \neg\sigma_r$ where no $\neg$ occurs in $\sigma_r$.
    By repeated applications of Lemma \ref{lem::flat-modalities}, the modality $\sigma_l\neg$ is equivalent to $\sigma'_l\neg$ where $\sigma'_l$ is one of the $14$ $\{\neg,\Box,\Diamond\}$-modalities of classical $\SF$.
    Therefore $\gsfc$ has at most $7 + 14 \cdot 7 = 105$ $\{\neg,\Box,\Diamond\}$-modalities.
\end{proof}

% \Acknowledgements{Acknowledgements such as thanks for reviewers' remarks can be put here.}

%%%%%%%%%%%%%%%%%%%%%%%%%%%%%%%%%%%%%%%%%%%%%%%%%%%%%%%%%%%%%%%%%%%%%%%%%%%%%%%%
\section{Conclusion}
%%%%%%%%%%%%%%%%%%%%%%%%%%%%%%%%%%%%%%%%%%%%%%%%%%%%%%%%%%%%%%%%%%%%%%%%%%%%%%%%
While in classical $\SF$ there are finitely many $\{\neg,\Box,\Diamond\}$-modalities, we prove that the situation is rather more complicated when we consider variations of $\SF$ based on intuitionistic propositional logic.
We close this paper with two directions for future work.

A first direction is to study the analogous problem for other non-classical modal logics.
This has been done in the classical setting, see, for example, Section~4.4 of Chellas~\cite{chellas1980textbook} which proves that the logics $\mathsf{KT4}$, $\mathsf{K5}$, $\mathsf{KD5}$, $\mathsf{K45}$, $\mathsf{KB4}$, $\mathsf{KD45}$, and $\mathsf{KT5}$ have finitely many $\{\neg,\Box,\Diamond\}$-modalities.
The main roadblock here would be the semantics for logics in the \emph{constructive} modal cube; for most of these logics, only proof systems have appeared in the literature~\cite{arisaka2015nestedsequents} and some of these coincide with their intuitionistic versions~\cite{pacheco2024ckb}.

A second direction is to study variations of $\SF$ and the logics mentioned above which are based on non-classical logics other than the intuitionistic propositional calculus.
For a short introduction to some of these variations, see~\cite{shawn2027nonclassicalmodallogic}.

\printbibliography

@inproceedings{alechina2001cs4,
	author    = {Natasha Alechina and
		Michael Mendler and
			Valeria {de Paiva} and
			Eike Ritter},
	editor    = {Laurent Fribourg},
	title     = {Categorical and {Kripke} Semantics for Constructive {S4} Modal Logic},
	booktitle = {Computer Science Logic, 15th International Workshop, {CSL} 2001. 10th
		Annual Conference of the EACSL, Paris, France, September 10-13, 2001,
		Proceedings},
	series    = {Lecture Notes in Computer Science},
	volume    = {2142},
	pages     = {292--307},
	publisher = {Springer},
	year      = {2001},
    doi       = {10.1007/3-540-44802-0_21},
}

@article{dosen1985stronger,
 author = {Do{\v s}en, Kosta},
 title = {Models for stronger normal intuitionistic modal logics},
 fjournal = {Studia Logica},
 journal = {Stud. Log.},
 issn = {0039-3215},
 volume = {44},
 pages = {39--70},
 year = {1985},
 doi = {10.1007/BF00370809},
 zbMATH = {4031639},
 Zbl = {0634.03015}
}

@article{font1986modality,
 author = {Font, Josep M.},
 title = {Modality and possibility in some intuitionistic modal logics},
 fjournal = {Notre Dame Journal of Formal Logic},
 journal = {Notre Dame J. Formal Logic},
 issn = {0029-4527},
 volume = {27},
 pages = {533--546},
 year = {1986},
 doi = {10.1305/ndjfl/1093636766},
 zbMATH = {4039855},
 Zbl = {0638.03017}
}

@article {parry1939modalities,
    AUTHOR = {Parry, William Tuthill},
     TITLE = {Modalities in the {\it {S}urvey} system of strict implication},
   JOURNAL = {J. Symbolic Logic},
  FJOURNAL = {The Journal of Symbolic Logic},
    VOLUME = {4},
      YEAR = {1939},
     PAGES = {137--154},
      ISSN = {0022-4812,1943-5886},
   MRCLASS = {02.0X},
  MRNUMBER = {809},
MRREVIEWER = {O.\ Frink},
       DOI = {10.2307/2268714},
}

@article{kuratowski1922closure,
 author = {Kuratowski, C.},
 title = {Sur l'op{\'e}ration {{\(A\)}} de l'analysis situs.},
 fjournal = {Fundamenta Mathematicae},
 journal = {Fundam. Math.},
 issn = {0016-2736},
 volume = {3},
 pages = {182--199},
 year = {1922},
 doi = {10.4064/fm-3-1-182-199},
 zbMATH = {2600910},
 JFM = {48.0210.04}
}

@article {gardner2008kuratowski,
    AUTHOR = {Gardner, B. J. and Jackson, M.},
     TITLE = {The {K}uratowski closure-complement theorem},
   JOURNAL = {New Zealand J. Math.},
  FJOURNAL = {New Zealand Journal of Mathematics},
    VOLUME = {38},
      YEAR = {2008},
     PAGES = {9--44},
      ISSN = {1171-6096,1179-4984},
   MRCLASS = {54A05},
  MRNUMBER = {2491682},
MRREVIEWER = {Ivan\ L.\ Reilly},
}

@article {ciraulo2025constructivekuratowski,
    AUTHOR = {Ciraulo, Francesco},
     TITLE = {Kuratowski's problem in constructive topology},
   JOURNAL = {J. Log. Anal.},
  FJOURNAL = {Journal of Logic and Analysis},
    VOLUME = {17},
      YEAR = {2025},
      number = {FDS1},
     PAGES = {27},
      ISSN = {1759-9008},
   MRCLASS = {54A05 (03F65 06A15 06D22)},
  MRNUMBER = {4869930},
       DOI = {10.4115/jla.2025.17.fds1},
}

@article {bellissima1989modalities,
    AUTHOR = {Bellissima, Fabio and Mirolli, Massimo},
     TITLE = {A general treatment of equivalent modalities},
   JOURNAL = {J. Symbolic Logic},
  FJOURNAL = {The Journal of Symbolic Logic},
    VOLUME = {54},
      YEAR = {1989},
    NUMBER = {4},
     PAGES = {1460--1471},
      ISSN = {0022-4812,1943-5886},
   MRCLASS = {03B45},
  MRNUMBER = {1026610},
MRREVIEWER = {Sergei\ N.\ Artemov},
       DOI = {10.2307/2274826},
}

@article {vanbenthem1976reduction,
    AUTHOR = {van Benthem, J. F. A. K.},
     TITLE = {Modal reduction principles},
   JOURNAL = {J. Symbolic Logic},
  FJOURNAL = {The Journal of Symbolic Logic},
    VOLUME = {41},
      YEAR = {1976},
    NUMBER = {2},
     PAGES = {301--312},
      ISSN = {0022-4812,1943-5886},
   MRCLASS = {02C10},
  MRNUMBER = {409111},
MRREVIEWER = {David\ Makinson},
       DOI = {10.2307/2272228},
}

@article {fitch1973reduction,
    AUTHOR = {Fitch, Frederic B.},
     TITLE = {A correlation between modal reduction principles and
              properties of relations},
   JOURNAL = {J. Philos. Logic},
  FJOURNAL = {Journal of Philosophical Logic},
    VOLUME = {2},
      YEAR = {1973},
    NUMBER = {1},
     PAGES = {97--101},
      ISSN = {0022-3611},
   MRCLASS = {02C10},
  MRNUMBER = {414319},
MRREVIEWER = {Curt\ Christian},
       DOI = {10.1007/BF02115611},
}

@article{mckinsey1944topology,
	author={{J.C.C.} McKinsey and A. Tarski},
	title={The algebra of topology},
	journal={Annals of Mathematics},
	series={2},
	volume={45},
    number={1},
	year={1944},
	pages={141--191},
    doi={10.2307/1969080},
}

@book {chellas1980textbook,
    AUTHOR = {Chellas, Brian F.},
     TITLE = {Modal logic},
      NOTE = {An introduction},
 PUBLISHER = {Cambridge University Press, Cambridge-New York},
      YEAR = {1980},
     PAGES = {xii+295},
      ISBN = {0-521-22476-4},
   MRCLASS = {03B45 (03-01)},
  MRNUMBER = {556867},
MRREVIEWER = {Robert\ P.\ McArthur},
}

@article{mendler2005constructive,
	title={Constructive {CK} for contexts},
	author={Mendler, Michael and {de Paiva}, Valeria},
	journal={Context Representation and Reasoning (CRR-2005)},
	volume={13},
	year={2005},
	publisher={CEUR Proceedings, Paris}
}

@article{wijesekera1990constructive,
	title={Constructive modal logics {I}},
	author={Wijesekera, Duminda},
	journal={Annals of Pure and Applied Logic},
	volume={50},
	number={3},
	pages={271--301},
	year={1990},
	publisher={Elsevier},
	doi = {10.1016/0168-0072(90)90059-B}
}

@phdthesis{simpson1994phdthesis,
	author    = {A. Simpson},
	title     = {The proof theory and semantics of intuitionistic modal logic},
	school    = {University of Edinburgh, {UK}},
	year      = {1994},
}

@article{balbiani2024variants,
  title={Constructive S4 modal logics with the finite birelational frame property},
  author={Balbiani, Philippe and Di{\'e}guez, Mart{\'\i}n and Fern{\'a}ndez-Duque, David and McLean, Brett},
  year={2024},
  url={https://arxiv.org/abs/2403.00201}
}

@inproceedings{balbiani2021variants,
  title={Some constructive variants of S4 with the finite model property},
  author={Balbiani, Philippe and Di{\'e}guez, Mart{\'\i}n and Fern{\'a}ndez-Duque, David},
  booktitle={2021 36th Annual ACM/IEEE Symposium on Logic in Computer Science (LICS)},
  pages={1--13},
  year={2021},
  organization={IEEE},
  doi={10.1109/LICS52264.2021.9470643},
}

@article{aguilera2026polytopological,
  title={Polytopological Semantics for Intuitionistic Modal Logics},
  author={Aguilera, Juan P and Fern{\'a}ndez-Duque, David and Pacheco, Leonardo},
  url={https://arxiv.org/abs/2604.23234},
  year={2026}
}

@inproceedings{girlando2023is4,
  author       = {Marianna Girlando and
                  Roman Kuznets and
                  Sonia Marin and
                  Marianela Morales and
                  Lutz Stra{\ss}burger},
  title        = {Intuitionistic {S4} is decidable},
  booktitle    = {38th Annual {ACM/IEEE} Symposium on Logic in Computer Science, {LICS}
                  2023, Boston, MA, USA, June 26-29, 2023},
  pages        = {1--13},
  publisher    = {{IEEE}},
  year         = {2023},
  doi          = {10.1109/LICS56636.2023.10175684},
  bibsource    = {dblp computer science bibliography, https://dblp.org}
}

@inproceedings{girlando2026is4mistake,
  author       = {Marianna Girlando and
                  Roman Kuznets and
                  Sonia Marin and
                  Marianela Morales and
                  Lutz Stra{\ss}burger},
  title        = {An Unfinished Story: Decidability of Intuitionistic S4},
  booktitle    = {8th Intuitionistic Modal Logic and Applications Workshop},
  year         = {2026},
  url          = {https://sonia-marin.github.io/imla26/imla-finals/imla-final15.pdf},
}

@article {caicedo2015bimodalgodellogic,
    AUTHOR = {Caicedo, Xavier and Rodr\'iguez, Ricardo Oscar},
     TITLE = {Bi-modal {G}\"odel logic over {$[0,1]$}-valued {K}ripke
              frames},
   JOURNAL = {J. Logic Comput.},
  FJOURNAL = {Journal of Logic and Computation},
    VOLUME = {25},
      YEAR = {2015},
    NUMBER = {1},
     PAGES = {37--55},
      ISSN = {0955-792X,1465-363X},
   MRCLASS = {03B45 (03B50)},
  MRNUMBER = {3365494},
MRREVIEWER = {Pedro\ Cabalar},
       DOI = {10.1093/logcom/exs036},
}

@article {rodriguez2021crispgodellogic,
    AUTHOR = {Rodríguez, Ricardo Oscar and Vidal, Amanda},
     TITLE = {Axiomatization of crisp {G}\"odel modal logic},
   JOURNAL = {Studia Logica},
  FJOURNAL = {Studia Logica. An International Journal for Symbolic Logic},
    VOLUME = {109},
      YEAR = {2021},
    NUMBER = {2},
     PAGES = {367--395},
      ISSN = {0039-3215,1572-8730},
   MRCLASS = {03B50 (03B45)},
  MRNUMBER = {4218797},
MRREVIEWER = {Giacomo\ Lenzi},
       DOI = {10.1007/s11225-020-09910-5},
}

@article{arisaka2015nestedsequents,
	author    = {Ryuta Arisaka and
		Anupam Das and
			Lutz Stra{\ss}burger},
	title     = {On Nested Sequents for Constructive Modal Logics},
	journal   = {Logical Methods in Computer Science},
	volume    = {11},
	number    = {3},
	year      = {2015},
	doi       = {10.2168/LMCS-11(3:7)2015},
}

@article{pacheco2024ckb,
  title={Collapsing Constructive and Intuitionistic Modal Logics},
  author={Pacheco, Leonardo},
  url={https://arxiv.org/abs/2408.16428},
  year={2024}
}

@incollection{shawn2027nonclassicalmodallogic,
  author    = "Søren Knudstorp and Igor Sedlar and Shawn Standefer",
  title     = "Non-Classical Logics",
  booktitle = "Modal Logic Today: A Handbook",
  year      = "to appear"
}

@article{servi1984axiomatizations,
	title={Axiomatizations for some intuitionistic modal logics},
	author={Gis{\`e}le {Fischer Servi}},
	journal={Rend. Sem. Mat. Univers. Politecn. Torino},
	volume={42},
	number={3},
	pages={179--194},
	year={1984}
}

\end{document}